\documentclass[12pt]{article}
\pdfoutput=1
\usepackage{amsmath,amsthm,amssymb,bbm,color,enumerate}
\usepackage{graphicx}
\usepackage{subcaption}
\usepackage[UKenglish]{babel}
\usepackage{lmodern}
\usepackage{microtype}
\usepackage{mathtools}
\usepackage{esint}
\usepackage{comment}
\usepackage{hyperref}
\usepackage{authblk}
\usepackage[utf8]{inputenc}

\mathtoolsset{showonlyrefs}  

\theoremstyle{plain}
\newtheorem{theorem}{Theorem}

\newtheorem*{theorem*}{Theorem}
\newtheorem{lemma}[theorem]{Lemma}
\newtheorem{corollary}[theorem]{Corollary}

\newtheorem{proposition}[theorem]{Proposition}

\theoremstyle{definition}
\newtheorem{definition}[theorem]{Definition}

\theoremstyle{remark}
\newtheorem{remark}[theorem]{Remark}

\newcommand*{\R}{\mathbb{R}}

\newcommand*{\N}{\mathbb{N}}

\newcommand*{\F}{\mathcal{F}}
\newcommand*{\E}{\mathbb{E}}
\newcommand*{\Prob}{\mathbb{P}}
\newcommand*{\Leb}{\mathcal{L}}

\newcommand*{\ol}[1]{\overline{#1}}
\newcommand{\sulut}[1]{\left( #1 \right)}
\newcommand{\parens}[1]{\left( #1 \right)}
\newcommand{\joukko}[1]{\left\{ #1 \right\}}

\newcommand{\der}{\mathrm{d}}

\newcommand{\ip}[2]{\left\langle#1,#2\right\rangle}

\DeclareMathOperator{\supp}{supp}

\makeatletter
\newcommand*{\mint}[1]{%
  \mint@l{#1}{}%
}
\newcommand*{\mint@l}[2]{%
  \@ifnextchar\limits{%
    \mint@l{#1}%
  }{%
    \@ifnextchar\nolimits{%
      \mint@l{#1}%
    }{%
      \@ifnextchar\displaylimits{%
        \mint@l{#1}%
      }{%
        \mint@s{#2}{#1}%
      }%
    }%
  }%
}
\newcommand*{\mint@s}[2]{%
  \@ifnextchar_{%
    \mint@sub{#1}{#2}%
  }{%
    \@ifnextchar^{%
      \mint@sup{#1}{#2}%
    }{%
      \mint@{#1}{#2}{}{}%
    }%
  }%
}
\def\mint@sub#1#2_#3{%
  \@ifnextchar^{%
    \mint@sub@sup{#1}{#2}{#3}%
  }{%
    \mint@{#1}{#2}{#3}{}%
  }%
}
\def\mint@sup#1#2^#3{%
  \@ifnextchar_{%
    \mint@sup@sub{#1}{#2}{#3}%
  }{%
    \mint@{#1}{#2}{}{#3}%
  }%
}
\def\mint@sub@sup#1#2#3^#4{%
  \mint@{#1}{#2}{#3}{#4}%
}
\def\mint@sup@sub#1#2#3_#4{%
  \mint@{#1}{#2}{#4}{#3}%
}
\newcommand*{\mint@}[4]{%
  \mathop{}%
  \mkern-\thinmuskip
  \mathchoice{%
    \mint@@{#1}{#2}{#3}{#4}%
        \displaystyle\textstyle\scriptstyle
  }{%
    \mint@@{#1}{#2}{#3}{#4}%
        \textstyle\scriptstyle\scriptstyle
  }{%
    \mint@@{#1}{#2}{#3}{#4}%
        \scriptstyle\scriptscriptstyle\scriptscriptstyle
  }{%
    \mint@@{#1}{#2}{#3}{#4}%
        \scriptscriptstyle\scriptscriptstyle\scriptscriptstyle
  }%
  \mkern-\thinmuskip
  \int#1%
  \ifx\\#3\\\else_{#3}\fi
  \ifx\\#4\\\else^{#4}\fi  
}
\newcommand*{\mint@@}[7]{%
  \begingroup
    \sbox0{$#5\int\m@th$}%
    \sbox2{$#5\int_{}\m@th$}%
    \dimen2=\wd0 %
    \let\mint@limits=#1\relax
    \ifx\mint@limits\relax
      \sbox4{$#5\int_{\kern1sp}^{\kern1sp}\m@th$}%
      \ifdim\wd4>\wd2 %
        \let\mint@limits=\nolimits
      \else
        \let\mint@limits=\limits
      \fi
    \fi
    \ifx\mint@limits\displaylimits
      \ifx#5\displaystyle
        \let\mint@limits=\limits
      \fi
    \fi
    \ifx\mint@limits\limits
      \sbox0{$#7#3\m@th$}%
      \sbox2{$#7#4\m@th$}%
      \ifdim\wd0>\dimen2 %
        \dimen2=\wd0 %
      \fi
      \ifdim\wd2>\dimen2 %
        \dimen2=\wd2 %
      \fi
    \fi
    \rlap{%
      $#5%
        \vcenter{%
          \hbox to\dimen2{%
            \hss
            $#6{#2}\m@th$%
            \hss
          }%
        }%
      $%
    }%
  \endgroup
}

\title{Optimal recovery of a radiating source with multiple frequencies along one line}

\author{Tommi Brander}
\affil{Norwegian University of Science and Technology, Department of Mathematical Sciences; tommi.brander@ntnu.no}
\affil{Technical University of Denmark, Department of Applied Mathematics and Computer Science}

\author{Joonas Ilmavirta}
\affil{University of Jyväskylä, Department of Mathematics and Statistics; joonas.ilmavirta@jyu.fi}

\author{Petteri Piiroinen}
\affil{University of Helsinki, Department of Mathematics and Statistics; petteri.piiroinen@helsinki.fi}

\author{Teemu Tyni}
\affil{University of Oulu, Research Unit of Mathematical Sciences; teemu.tyni@helsinki.fi}

\begin{document}

\maketitle


\begin{abstract}
We study an inverse problem where an unknown radiating source is observed with collimated detectors along a single line and the medium has a known attenuation.
The research is motivated by applications in SPECT and beam hardening.
If measurements are carried out with frequencies ranging in an open set, we show that the source density is uniquely determined by these measurements up to averaging over levelsets of the integrated attenuation.
This leads to a generalized Laplace transform.
We also discuss some numerical approaches and demonstrate the results with several examples.
\end{abstract}








\section{Introduction}

We consider the following one-dimensional inverse problem:
The intensity of radiation from an unknown source in a known medium is measured at multiple frequencies.
How uniquely does this determine the density of the source?
A more detailed description of the model is given in section~\ref{sec:model} below.

This physical problem boils down to the following mathematical question on the interval $I=(a,b)$:
Given a function $p\colon I\to\R$, does the knowledge of the function
\begin{equation}
\label{eq:Dlambda}
D(\lambda)
=
\int_a^b\lambda^{p(x)}\rho(x)\der x
\end{equation}
for~$\lambda$ in an open set $U\subset(0,1)$ determine the function $\rho\colon I\to\R$ uniquely?
Uniqueness of~$\rho$ depends on the properties of~$p$ in a peculiar way.

We denote by $P\colon L^2(I)\to L^2(I)$ the unique projection onto $L^2(I,\sigma(p))$, the subspace of $L^2(I)$ consisting of $\sigma(p)$-measurable functions.
Here~$\sigma(p)$ is the smallest sigma-algebra on~$I$ which makes~$p$ measurable and contains sets of zero Lebesgue measure.
This operator satisfies for any $\rho\in L^2(I)$ that the function~$P\rho$ is $\sigma(p)$-measurable and for every $\sigma(p)$-measurable $A\subset I$ we have
\begin{equation}
\int_A P\rho(x)\der x
=
\int_A\rho(x)\der x.
\end{equation}
If~$I$ has measure one, the operator is the conditional expectation~$\E\left[ \rho | \sigma(p) \right] $; for more on conditional expectation we refer to the books~\cite[chapter~2]{Dellacherie:Meyer:1978}\cite[chapter~5]{Kallenberg:1997}.
For more details on the operator, see section~\ref{sec:def-not}.

The conclusion is that the data~$D$ does not in general determine~$\rho$, but it does determine the projection~$P\rho$ and nothing more.
We have unique determination precisely when~$\sigma(p)$ is the whole Lebesgue algebra.
This is formulated precisely in our main theorem, which we prove in two different languages; by a measure-theoretical approach in section~\ref{sec:proof} and a probabilistic one in section~\ref{sec:prob}:

\begin{theorem}
\label{thm:main}
Suppose $\rho \in L^2(I)$ and $p \in L^\infty(I)$.
Let $U\subset(0,1)$ be a nonempty open set.
Then the following are equivalent:
\begin{enumerate}
    \item The function~$D(\lambda)$ defined in~\eqref{eq:Dlambda} vanishes for all $\lambda\in U$.
    \item $P\rho=0$, that is, the function~$\rho$ is orthogonal to $L^{2}(I,\sigma(p))$.
    \item For all $A \in \sigma(p)$ it holds that $\int_A \rho(x) \der x = 0$.
\end{enumerate}
\end{theorem}
The probabilistic approach also works with $\rho \in L^1(I)$ and with $U$ a sequence with a cluster point.
The next corollary follows immediately.

\begin{corollary}
If the functions~$D_1$ and~$D_2$ arise from the functions~$\rho_1$ and~$\rho_2$ by~\eqref{eq:Dlambda}, then $D_1=D_2$ if and only if $P\rho_1=P\rho_2$.

The linear map $\rho\mapsto D$ is injective if and only if~$\sigma(p)$ is the whole Lebesgue algebra. This happens, in particular, when~$p$ is injective.
\end{corollary}

The function $D\colon U\to\R$ defined by~\eqref{eq:Dlambda} is the data, and it depends linearly on the unknown function~$\rho$.
Therefore it follows from the theorem that a function $\rho\in L^2(I)$ is determined by~$D$ up to an element of the space $L^{2}(I,\sigma(p))^\perp\subset L^2(I)$, which is the kernel of the linear operator $\rho\mapsto D$.
Another way of viewing this is that the push-forward of $\rho \der x$ under $p$ is recovered.%

The theorem completely characterizes what can be said about~$\rho$, given~$p$ and~$D$.
If the attenuation coefficient is strictly positive (see section~\ref{sec:model}), then~$p$ is strictly increasing and the source density~$\rho$ is determined fully uniquely.

For a concrete situation where the result applies, consider a medium composed of a single material.
That material has been CT scanned so as to determine its density~$\beta(x)$.
Then suppose another substance is added, and it emits radiation on a broad spectrum.
The density of this new material on a single line can be monitored by measuring intensity of radiation coming along that line at multiple frequencies with a collimated detector.
If $\beta>0$, this information determines the density~$\rho$ of the source uniquely.
To monitor the intensity of the source as a function of time after the initial CT scan, it is sufficient to use a single line for measurements.
This idea is similar to that used in SPECT and PET where a CT scan is first needed to map the attenuation before something else is used to image the radiating source.
For more on related imaging modalities, see section~\ref{sec:background}.

\subsection{The model}
\label{sec:model}

Suppose the source of radiation is $f(\omega,x)$ and the attenuation coefficient is $\mu(\omega,x)$.
The frequency~$\omega$ takes values in an open set $U\subset(0,\infty)$ and $x\in I=(a,b)$.
The measurement is the intensity of radiation at~$a$, which is given by
\begin{equation}
M(\omega)
=
\int_a^b
e^{-\int_a^x\mu(\omega,y)\der y}
f(\omega,x)
\der x
\end{equation}
according to Beer-Lambert's law~\cite{Ingle:Crouch:1988}.
The physical problem is to recover the source~$f$ using the measurement~$M(\omega)$ for a large number of frequencies~$\omega$.

To do this, structural assumptions on~$\mu$ and~$f$ are needed.
We assume that they both factor: $\mu(\omega,x)=\alpha(\omega)\beta(x)$ and $f(\omega,x)=\epsilon(\omega)\rho(x)$.
The functions $\beta$ and $\rho$ can be regarded as the spatial densities of the absorbent and the source.
The functions $\alpha$ and $\epsilon$ correspond to ``spectral densities'' and depend on the physical process behind attenuation and emission.
We note that similar assumptions have been made in an earlier study~\cite[remark 1]{Gourion:Noll:2002}.

If both the absorbent and the source are composed of a single material, this factorization is well-justified.
The rate of absorption or emission is directly proportional to the density, and the functions $\alpha$ and $\epsilon$ are simply the coefficients of proportionality which may well --- and generally do --- depend on frequency.
If there are multiple materials, the frequency dependence can be different for the different materials, and the overall attenuation coefficient and source no longer factorize.

We introduce two auxiliary functions, $\phi(\omega)=e^{-\alpha(\omega)}$ and $p(x)=\int_a^x\beta(y)\der y$.
We work on a frequency range where $\alpha(\omega)>0$, and so $\phi(\omega)=\lambda\in(0,1)$.
We assume that there is a function $\eta\colon U\to\R$ so that $\phi(\eta(\lambda))=\lambda$ for all $\lambda\in U$.
That is,~$\eta$ is a right inverse of~$\phi$, and it exists for some interval~$U$ if, for example,~$\alpha$ is continuously differentiable and non-constant.

With these assumptions the measurement~$M(\omega)$ may be processed to yield the data
\begin{equation}
D(\lambda)
=
\frac{M(\eta(\lambda))}{\epsilon(\eta(\lambda))}
\end{equation}
used in the equation~\eqref{eq:Dlambda} as a function of~$\lambda$.

\subsection{Background}
\label{sec:background}

In the inverse problem presented above, the goal is to reconstruct the source when the attenuation is known.
Well-studied imaging modalities with the same goal are instances of emission computer tomography.
We mention single-photon emission computed tomography (SPECT) and positron emission tomography (PET).

In both SPECT and PET a radioactive substance is injected into the target.
The substance decays and emits gamma radiation, which is detected outside the target.
From these measurements, one tries to reconstruct the location of the radioactive substance.
In SPECT, the substance emits single gamma-ray photons, which are then detected.
In PET, it emits positrons, which soon combine with an electron, shooting two gamma ray photons to opposite directions; this pair is then detected~\cite{Lewitt:Matej:2003,Gourion:Noll:2002,Welch:Gullberg:Christian:Li:Tsui:1995}.
In SPECT, some radioactive substances radiate at several frequencies~\cite{Welch:Gullberg:Christian:Li:Tsui:1995}.

In both PET and SPECT the reconstruction is improved if the anatomy of the target is known a~priori~\cite{Gourion:Noll:2002}.
A CT scan is a common approach.
Once the anatomy is known, the values of attenuation can be recovered based on known values.
We likewise assume a known attenuation and an unknown source; hence, we are investigating a model of one-dimensional multispectral SPECT/PET.

When a multispectral X-ray beam passes through a material, low-energy photons are typically attenuated more strongly than high-energy ones, which changes the frequency profile of the beam. This is called beam hardening~\cite{Brooks:DiChiro:1976}.
Our model is consistent with this phenomenon -- we assume the factorization $\mu(\omega,x) = \alpha(\omega)\beta(x)$, where the $\alpha$ is the dependency of the attenuation on the frequency.
Our model can not only take beam hardening into account but make use of it.

Multispectral (also multi-energy, multichromatic, spectral, spectroscopic, energy-selective, energy-sensitive, energy discrimination, or colour) X-ray tomography started with the work of Alvarez and Macovski~\cite{Alvarez:Macovski:1976}.
Incoming photons are classified in a number, often two, of energy bins according to their energy by the photon-counting detectors, after which one can make separate reconstructions at different energy levels or attempt a joint reconstruction from all the available information.
The two energy levels are especially natural due to Compton and photoelectric effects~\cite{Alvarez:Macovski:1976}.
The main challenges are that the measurement devices are expensive and that the smaller amount of radiation leads to worse reconstructions at every energy level~\cite{Taguchi:Iwanczyk:2013}.
For more on multispectral X-ray tomography we refer to the reviews~\cite{Lehmann:Alvarez:1986,McCollough:Leng:Yu:Fletcher:2015}.

Recovering source terms in an attenuating medium can be formulated as the attenuated or exponential X-ray or Radon transform~\cite[section 8.8]{Deans:2007}\cite{Ilmavirta:Monard}\cite[section VI-C]{Lewitt:Matej:2003}\cite[chapter~8]{Sharafutdinov:1994}, which is a multidimensional theory that only uses measurements at a single frequency.
The attenuated Radon transform has an explicit inversion formula~\cite{Novikov:2002,Natterer:2001,Finch:2003}.
The present reconstruction theory is one-dimensional, and the X-ray transform with or without attenuation is never injective in one dimension.
The key is to use several different attenuations or weights along the fixed line, and in our case this is achieved by using a large number of frequencies.

A different way of using several weights is to study the momentum ray transform~\cite{Sharafutdinov:1994,Krishnan:Venkateswaran:Manna:Sahoo:Sharafutdinov:2019}.
In that problem, one defines the momentum ray transform for points $(x,\xi) \in T S^{d-1}$ by
\begin{equation}
\sulut{I^kf}\sulut{x,\xi} = \int_{-\infty}^\infty t^k \ip{f(x+t\xi)}{\xi^m} \der t
\end{equation}
for suitable tensor fields~$f$, and tries to recover~$f$ from the integrals indexed by all points $(x,\xi)$ and sufficient number of powers~$k$.
The usual X-ray transform~$I^0$ of a tensor field only determines the field up to a gauge, but using moments up to the order of the tensor field determines it uniquely.
For recent results in tensor tomography, we refer to~\cite{Paternain:Salo:Uhlmann:2014,Paternain:Salo:Uhlmann:2015,Ilmavirta:Monard}, and we also mention the classical book of Sharafutdinov~\cite{Sharafutdinov:1994}.
In the same spirit of using moments, in the works~\cite{Choi:Ginting:Jafari:Mnatsakanov,Milanfar:1993,Milanfar:Karl:Willsky:1996} the moments of noisy measurements are related to the moments of the unknown density function.

The Hausdorff moment problem~\cite{Schmudgen:2017,Widder:1946} asks: Given a sequence~$(s_n)$, does there exist a measure~$\mu$ such that, for all $n \in \N$,
\begin{equation}
s_n = \int_0^1 x^n \der \mu ?
\end{equation}
If it exists, is it unique?
In our problem, we know that the measure $\rho(x) \der x$ exists and want to understand its uniqueness, or, in the language of moment problem literature, determinacy.
More fundamentally, whereas in the moment problem one couples the measure with polynomials, we use polynomials of a function~$p$, which might not be continuous or injective.

Our main theorem turns out to be similar to an inverse problems result for the variable exponent $p(\cdot)$-Laplacian~\cite{Brander:Winterrose:2019}.
The methods are quite similar, but in the variable exponent case the equivalent of the measurements~$\lambda$ are explicitly related to a quantity~$K_\lambda$, which cannot be expressed analytically as a function of~$\lambda$, and the measurements are, using the notations in this paper,
\begin{equation}
\int_I \rho(x) K_\lambda^{p(x)/(p(x)-1)}.
\end{equation}
The intermediate quantity~$K_\lambda$ causes complications in the proofs, but the final results are similar to the present ones.
In another result the exponent~$p$ has been recovered~\cite{Brander:Siltakoski}, which suggest that it might be possible to recover the attenuation from known sources in the present model.

Our problem can also be seen as inverting a generalized Laplace transform~\cite{Widder:1946}.
Namely, if~$\rho$ is continuously differentiable and satisfies $\rho'>0$, then after changing the variable of integration in~\eqref{eq:Dlambda} from~$x$ to $y=p(x)$ and writing $\lambda=e^{-\alpha}$ (see section~\ref{sec:model}), the data can be rewritten as
\begin{equation}
D(\lambda)
=
\tilde D(\alpha)
=
\int_{\tilde a}^{\tilde b}e^{-\alpha y}\tilde\rho(y)\der y
,
\end{equation}
where $\tilde a=p(a)$, $\tilde b=p(b)$, and
$\tilde\rho(y)=\rho(p^{-1}(y))/p'(p^{-1}(y))$.
Therefore~$\tilde D$ is the Laplace transform of~$\tilde\rho$.
However, when~$p$ is less well-behaved, such reduction to Laplace transform does not work.
We also point out that putting $p(x)=x$ in theorem~\ref{thm:main} implies that if~$\rho$ is a compactly supported~$L^2$ function whose Laplace transform vanishes on an interval, then $\rho=0$.
This corollary is of course not new~\cite[chapter 2, section 6]{Widder:1946}.

One can also think of equation~\eqref{eq:Dlambda} as a Fredholm integral equation of the first kind for the compact operator $\widetilde{D}\colon L^2(I)\to L^2(U)$ given by $\widetilde{D}(\rho)(\lambda) \coloneqq \int_a^b\lambda^{p(x)}\rho(x)\der x$.

\subsection{Discussion}

If~$p$ is piecewise constant --- which corresponds to the attenuation being a sum of delta functions corresponding to thin absorbent films --- then the data~$D$ determines the average of~$\rho$ on every piece.
Nothing else is determined, and the averages of~$\rho$ over the levelsets of~$p$ is optimal information.

If~$p$ is strictly increasing --- which corresponds to strictly positive attenuation --- then heuristically the levelsets are points and the averages should determine~$\rho$ uniquely.
This is indeed the case, as the sigma-algebra generated by~$p$ is the full Lebesgue algebra.

Non-negativity of the attenuation coefficient implies that~$p$ is increasing.
The model can be extended to the case where $p\colon(a,b)\to[0,\infty]$, where $p=\infty$ corresponds to the intensity being attenuated all the way to zero.
If $p=\infty$ on some subinterval $[b',b)\subset(a,b)$, then the data tells nothing about~$\rho$ on this subinterval as $\lambda^p=0$ on this set.
Therefore we may restrict the problem to the interval $(a,b')$ with no loss of data.
It is thus reasonable to assume that on the interval of interest $p<\infty$.

Very weak regularity assumptions on the attenuation coefficient~$\beta$ are sufficient.
If $\beta\in L^1(I)$, then~$p$ is absolutely continuous, which is more than enough for our theorem.

Physically~$p$ is increasing (since $\beta\geq0$), but for mathematical purposes it may be any bounded measurable function.
If $I=[-1,1]$ and $p(x)=x^2$, then the sigma-algebra~$\sigma(p)$ generated by~$p$ consists of subsets of~$I$ that are symmetric with respect to reflection up to an error of measure zero.
Then the data~$D$ determines the symmetric part of~$\rho$ uniquely but does not constrain the antisymmetric part at all.

The rough idea of the proof is that linear combinations of functions $x\mapsto\lambda^{p(x)}$ indexed by $\lambda\in U$ can be used to approximate any $\sigma(p)$-measurable function $I\to\R$.
Thus $D=0$ is equivalent with~$\rho$ being perpendicular to the subspace of these functions.
This implies that variations of~$\rho$ within levelsets of~$p$ are undetectable.
This phenomenon is well-illustrated by a simple observation which we provide next.

\begin{proposition}
Suppose $c \in \R$ and let $\ol \rho \in L^{2}\sulut{I}$ be such that $\supp\sulut{\ol \rho} \subseteq p^{-1}\sulut{\joukko{c}}$ and
\begin{equation}
\int_{p^{-1}\sulut{\joukko{c}}} \ol \rho(x) \der x = 0.
\end{equation}
Then, for all $\lambda \ge 0$,
\begin{equation}
\int_{I} \sulut{\ol \rho(x) + \rho(x)} \lambda^{p(x)} \der x = \int_{I} \rho(x) \lambda^{p(x)} \der x.
\end{equation}
\end{proposition}

The proof is a straightforward calculation.

\subsection*{Acknowledgements}
T.B.\ was partially funded by grant no.\ 4002-00123 from the Danish Council for Independent Research | Natural Sciences and partially by the Research Council of Norway through the FRIPRO Toppforsk project ''Waves and nonlinear phenomena''.
J.I.\ was supported by the Academy of Finland (decision~295853).
T.T.\ was supported by the Academy of Finland (application number~312123, Centre
of Excellence of Inverse Modelling and Imaging 2018–2025).

We would like to thank Alexander Meaney for a discussion and valuable references and the anonymous referees for insightful comments and suggestions.

\section{Main results}
\label{sec:proof}
In this section we specify notation and prove the main theorem.
An alternative stochastic proof can be found in section~\ref{sec:prob}.

\subsection{Definitions and notation}
\label{sec:def-not}

We denote by~$\Leb$ the Lebesgue sigma-algebra on $I=(a,b)$.
We always use the Lebesgue measure~$\der x$.
The sigma-algebra~$\sigma(p)$ generated by a function $p\colon I\to\R$ is also a sigma-algebra on~$I$, and we define it as the smallest sigma-algebra so that~$p$ is measurable and sets of zero Lebesgue outer measure are measurable.
For more on sigma-algebras generated by sets and functions, see the books~\cite[chapter~1, definition~5]{Dellacherie:Meyer:1978}\cite[chapter~1]{Kallenberg:1997}.

The following lemma states that~$\sigma(p)$ does not depend on the representative of~$p$.

\begin{lemma}
If $g = h$ almost everywhere, then $\sigma(g)=\sigma(h)$.
\end{lemma}
\begin{proof}
Write as $\tilde \sigma(g)$ the preimage $g^{-1}(\sigma(\R))$, which is a sigma-algebra~\cite[lemma 1.3]{Kallenberg:1997}.
Suppose $A \in \tilde\sigma(g
)$.
There then exists a measurable $B \subset \R$ such that $A = g^{-1}(B)$.
Now
\begin{equation}
    \sulut{A \setminus h^{-1}(B)} \cup \sulut{h^{-1}(B) \setminus A} \subseteq \joukko{x \in I ; g(x) \neq h(x)},
\end{equation}
which is a null set.
Because $A = h^{-1}(B)$ up to null sets, we have $A \in \sigma(h)$.

If $A' \in \sigma(g)$, then it is equal to some $A \in \tilde \sigma(g)$ up to null sets, and by the above reasoning $A' \in \sigma(h)$.
\end{proof}

The space $L^q(I,\sigma(p))$ is the space of $\sigma(p)$-measurable functions in~$L^q(I)$ up to almost everywhere equality.
If $\sigma(p)=\Leb$ (the Lebesgue sigma-algebra), then $L^q(I,\sigma(p))=L^q(I)$.

Since $L^2\sulut{I,\mathcal{L}}$ is a complete Hilbert space and $L^2\sulut{I,\sigma(p)}$ a closed convex set, we can define the following unique orthogonal projection.
\begin{definition}
The mapping~$P$ is the orthogonal projection $P \colon L^2\sulut{I,\mathcal{L}} \to L^2\sulut{I,\sigma(p)}$.
\end{definition}

\subsection{Lemmas}

\begin{lemma}
\label{lemma:proj_to_int}
For all $A \in \sigma(p)$ the projection~$P$ satisfies
\begin{equation}
\int_A f \der x = \int_A Pf \der x.
\end{equation}
\end{lemma}

\begin{proof}
Since $A \in \sigma(p)$, the characteristic function~$\chi_A$ is in $L^2\sulut{I,\sigma(p)}$.
On the other hand, $f-Pf \in \sulut{L^2\sulut{I,\sigma(p)}}^\perp$.
Hence,
\begin{equation}
\int_A f \der x = \int_I \chi_A \sulut{Pf + (f-Pf)} \der x = \int_I \chi_A Pf \der x = \int_A Pf \der x.
\end{equation}
\end{proof}

\begin{lemma}
\label{lemma:int_to_proj}
If a mapping~$Q \colon L^2\sulut{I,\mathcal{L}} \to L^2\sulut{I,\sigma(p)}$ satisfies, for all $A \in \sigma(p)$,
\begin{equation}
\int_A f \der x = \int_A Qf \der x,
\end{equation}
then it is the orthogonal projection onto $L^2\sulut{I,\sigma(p)}$.
\end{lemma}

\begin{proof}
The map~$Q$ is a linear projection by definition.

For the characteristic function~$\chi_A$ of any $A \in \sigma(p)$ we have
\begin{equation}
\int_I \chi_A (f-Qf) \der x = 0,
\end{equation}
whence $f-Qf$ is orthogonal to $L^2\sulut{I,\sigma(p)}$, since any function there can be approximated by measurable step functions.
Because the range of~$Q$ is $L^2\sulut{I,\sigma(p)}$ and~$f$ is arbitrary, this shows orthogonality.
\end{proof}

Recall the data~$D \colon \R_+ \to \R$,
\begin{equation}
D(\lambda)
=
\int_a^b\lambda^{p(x)}\rho(x)\der x.
\end{equation}
\begin{lemma}
\label{lma:polynomi}
Suppose $\lambda_0$ is an interior point of an open set~$U \subset [0,1]$.
If $D(\lambda)=0$ for all $\lambda\in U$, then
\begin{equation}
\label{eq:poly-int}
\int_a^b r(p(x)) \lambda_0^{p(x)} \rho(x) \der x=0
\end{equation}
for all polynomial functions~$r$.
\end{lemma}

\begin{proof}
The function $D$ is smooth.
In fact, it is complex analytic in a neighborhood of~$U$ as a consequence of Morera's theorem~\cite[theorem 10.17]{Rudin:2006}.

For any natural number~$k$ applying the operator $\lambda \frac{\der}{\der \lambda}$ to~$D$ a total of $k$~times gives
\begin{equation}
\parens{\lambda \frac{\der}{\der \lambda}}^k D(\lambda) = \int_a^b  \lambda^{p(x)}\rho(x) \sulut{p(x)}^k \der x.
\end{equation}
Derivatives of all orders vanish at $\lambda=\lambda_0$, whence
\begin{equation}
\label{eq:vv1}
\int_a^b  \lambda_0^{p(x)}\rho(x) \sulut{p(x)}^k \der x =0
\end{equation}
for all~$k$.
Finite linear combinations of these integrals give equation~\eqref{eq:poly-int}.
\end{proof}

The identity~\eqref{eq:poly-int} holds for any function~$r$ that can be approximated by polynomials in a suitable sense.
We will next study this approximation.

The following multiplicative system theorem is a version of the monotone class theorem~\cite[theorem~1.1]{Kallenberg:1997} written in terms of functions, rather than sets~\cite[chapter~1, theorems~19--21]{Dellacherie:Meyer:1978}.

\begin{lemma}[{Multiplicative system theorem, \cite[chapter~1, theorem~21]{Dellacherie:Meyer:1978}}]
\label{lma:multisystem}
Suppose~$H$ is a vector space of real-valued bounded measurable functions on a measurable space~$X$.
Suppose~$H$ contains constant functions and is closed under the pointwise convergence of uniformly bounded increasing sequences of functions.
Let $E \subseteq H$ be closed under pointwise multiplication, and let~$\mathcal{G}$ be the sigma-algebra generated by~$E$.

Then~$H$ contains all bounded $\mathcal{G}$-measurable functions.
\end{lemma}

We note that the multiplicative system theorem considers a vector space of functions~$H$ and a subset~$E$, whereas we operate with Lebesgue spaces of equivalence classes of functions.

We consider the set of functions
\begin{equation}
\label{eq:E-def}
\tilde E
=
\{
\lambda_0^{p(\cdot)}
r(p(\cdot))
\,;\,r\text{ polynomial}
\}
\subset
L^\infty(I)
.
\end{equation}

Recall that~$\sigma(p)$ is the smallest sigma-algebra
that makes~$p$ measurable and contains sets of measure zero.

\begin{lemma}
\label{lma:poly-closure}
The closure of~$\tilde E$ in the $L^2(I)$-norm is $L^2(I,\sigma(p))$.
\end{lemma}

The proof is by Nathaniel Eldredge~\cite{Eldredge:2018}, and similar to proofs in variable exponent Calderón's problem~\cite[lemmas 18 and 27]{Brander:Winterrose:2019}.
We omit the space~$L^2$ from the notation of the closure.


\begin{proof}[Proof of lemma~\ref{lma:poly-closure}]
The function $\lambda_0^{p(\cdot)}$ is $\sigma(p)$-measurable and bounded away from zero and infinity.
Therefore it may be ``divided out'' and it suffices to prove the lemma in the case $\lambda_0=1$.
We write as~$E$ the space of polynomials of~$p$.

Every continuous function of~$p$ is $\sigma(p)$-measurable, so the equivalence classes of the functions in~$E$ form a subspace of $L^2\sulut{I, \sigma(p)}$.
The space $L^2\sulut{I, \sigma(p)}$ is closed in~$L^2(I)$, since a converging sequence in~$L^2$ has a pointwise almost everywhere converging subsequence~\cite[theorem~3.12]{Rudin:2006}, and the limit of such a subsequence is, as the limit of measurable functions, measurable with respect to the same sigma-algebra.
Hence, we have $\ol{E} \subseteq L^2\sulut{I, \sigma(p)}$, where we understand~$\ol E$ as a space of equivalence classes of functions.

For the other direction, $L^2\sulut{I, \sigma(p)} \subseteq \ol{E}$, we start by considering the vector space~$H=\{f\,;\, [f]\in\ol{E}\cap L^\infty\sulut{I}\}$, which consists of all functions whose equivalence classes are in $\ol{E}\cap L^\infty\sulut{I}$.
It satisfies all the assumptions of the multiplicative system theorem:
\begin{itemize}
\item Constant functions are bounded and polynomials of the function~$p$.
\item If a uniformly bounded sequence converges pointwise, then the sequence of the squared absolute values of the functions also does so, as do their equivalence classes.
By monotone convergence, the~$L^p$ norms of the sequence converge.
\end{itemize}
We have that all representatives of~$E$ are in~$H$ and~$E$ is closed under pointwise multiplication.
By multiplicative system theorem (lemma~\ref{lma:multisystem}), $H$ contains all bounded $\sigma(p)$-measurable functions, whence $L^2\sulut{I, \sigma(p)} \cap L^\infty\sulut{I} \subseteq \ol{E}\cap L^\infty\sulut{I}$.

Consider a (not necessarily bounded) equivalence class of functions~$h \in L^2\sulut{I, \sigma(p)}$ and define
\begin{equation}
h_n(x) = \max\sulut{-n,\min\sulut{h,n}}.
\end{equation}
Then, for all $n \in \N$,  $h_n \in  H \subset \ol{E}$ and $h_n \to h$ in $L^2(I)$ as $n \to \infty$, so $h \in \ol{E}$.
\end{proof}

We note that we use a slightly different formalism than the paper of Brander and Winterrose~\cite{Brander:Winterrose:2019}; here,~$\sigma(p)$ is a completion of a sigma-algebra, whereas they instead use a mapping $[g]_{\sigma_p} \to [g]_{\Leb}$, which maps equivalence classes without the completion into equivalence classes with it.
This leads to superficial differences in the proof of lemma~\ref{lma:poly-closure} when compared to the similar proofs of \cite[lemmas 18 and 27]{Brander:Winterrose:2019}.

\subsection{Proof of theorem \ref{thm:main}}

We are now ready to prove our main result.
Let us first recall equation~\eqref{eq:Dlambda}:
\begin{equation*}
D(\lambda)
=
\int_a^b\lambda^{p(x)}\rho(x)\der x.
\end{equation*}
We also recall for convenience the main theorem, which is stated as:
\begin{theorem*}[Theorem \ref{thm:main}] 
Suppose $\rho \in L^2(I)$ and $p \in L^\infty(I)$.
Let $U\subset(0,1)$ be a nonempty open set.
Then the following are equivalent:
\begin{enumerate}
    \item The function $D(\lambda)$ vanishes for all $\lambda\in U$. 
    \item $P\rho=0$, that is, the function $\rho$ is orthogonal to $L^{2}(I,\sigma(p))$.
    \item For all $A \in \sigma(p)$ it holds that $\int_A \rho(x) \der x = 0$.
\end{enumerate}
\end{theorem*}

Now we are ready to prove it.

\begin{proof}[Proof of theorem~\ref{thm:main}]
1$\implies$2: 
By lemma~\ref{lma:polynomi} the~$L^2$ inner product between~$\rho$ and all polynomials of~$p$ is zero.
This implies that~$\rho$ is orthogonal to the~$L^2$ closure of the space of polynomials of~$p$, which, by lemma~\ref{lma:poly-closure}, is $L^2(I,\sigma(p))$.
Hence the projection is also zero.

2$\implies$1:
Let $\lambda \ge 0$. Since $\lambda^{p(x)}$ is $\sigma(p)$-measurable and bounded, we have $\lambda^{p(x)} \in L^2\sulut{I,\sigma(p)}$.
By orthogonality the $L^2$-inner product in equation~\eqref{eq:Dlambda} is zero.

2$\iff$3: This is lemmas~\ref{lemma:proj_to_int} and \ref{lemma:int_to_proj}.

If~$p$ injective, then $\sigma(p)=\Leb$ and injectivity of $\rho\mapsto D$ follows.
\end{proof}


\section{Probabilistic interpretation}
\label{sec:prob}

In this section we prove the main theorem in probabilistic language.
We make frequent use of the basic properties of conditional measures~\cite[section~5, especially theorem~5.1]{Kallenberg:1997} \cite[chapter 2, number~41]{Dellacherie:Meyer:1978}.
Let $\F$ be the sigma-algebra of Lebesgue-measurable sets on $\R$ (sometimes restricted to $I$ without changing the notation) and let $\Prob(A) = \frac{\der x}{b-a}$ be the Lebesgue measure rescaled into a probability measure.
Then $(I, \F, \Prob)$ is a probability space.
We use the Iverson bracket notation
\begin{equation}
[p \in A]
=
\begin{cases}
1, & \text{ if } p(x) \in A \\
0, & \text{ if } p(x) \notin A.
\end{cases} 
\end{equation}

\begin{theorem}
Suppose $\rho \in L^1\parens{ I, \F, \Prob }$ and $p \in L^\infty \parens{ I, \F, \Prob } $.
Let $U \subset \R_+$ be a nonempty open set and write $D\parens{\lambda} = \E \lambda^p \rho $.

Then the following are equivalent:
\begin{enumerate}
    \item For all $\lambda \in U$ we have $D(\lambda) = 0$.   \label{tfae:1}
    \item $\E \parens{ \rho | p } = 0 $ holds $\Prob$-a.s.  \label{tfae:2}
    \item For all $A \in \F$ we have $\E [p \in A] \rho = 0$.   \label{tfae:3}
\end{enumerate}
\end{theorem}

\begin{remark}
In \ref{tfae:1} in the theorem, it is sufficient that there exists a countably infinite set of frequencies $\lambda$ with a cluster point $\lambda_0$ such that $D(\lambda) = 0$.
\end{remark}
    
We prove the theorem with the remark included, as the remark is stronger than the theorem.
\begin{proof}
We show that the first and the second condition are equivalent and that the second and the third condition are equivalent.

\paragraph{\ref{tfae:2} $\implies$ \ref{tfae:1}:}
For all $\lambda \in \R_+$
\begin{equation}
    \begin{split}
        D(\lambda)  &= \E \lambda^p \rho = \E \parens{ \E \parens{ \lambda^p \rho | p } }   \\
                    &= \E \parens{ \lambda^p \E \parens{ \rho | p } } = 0.
    \end{split}
\end{equation}

\paragraph{\ref{tfae:2} $\iff$ \ref{tfae:3}:}
Define $\xi = \E \parens{ \rho | p }$.
Then, by definition, $\xi$ is the unique $\sigma(p)$-measurable random variable in $L^1\parens{ I, \F, \Prob }$ satisfying
$
\E[p \in A] \xi = \E [p \in A] \rho.
$
Now, \ref{tfae:2} implies \ref{tfae:3}:
\begin{equation}
    \E \parens{ [p \in A] \rho } = \E \parens{ [p \in A]\E\parens{ \rho | p } } = 0.
\end{equation}
On the other hand, \ref{tfae:3} implies \ref{tfae:2}. For all $A \in \F$
\begin{equation}
    \E[p \in A] \rho = 0 = \E[p \in A] \cdot 0.
\end{equation}
Thus, by uniqueness of $\xi$, we have $\xi = 0$ as a random variable, which is claim~\ref{tfae:2}.

\paragraph{\ref{tfae:1} $\implies$ \ref{tfae:2}:}
Let $V$ be a countably infinite set of points $\mu$ with cluster $\mu_0$.
Write $\mu_j = \log \lambda_j$ and suppose they converge to $\mu_0$ as a strictly increasing or a strictly decreasing sequence (by taking a subsequence without changing the notation).
Since $p$ is bounded, we can define $\widetilde D$ by $\widetilde D(\mu) = D(e^\mu) = 0$ for all $\mu \in V$.


By using $j+1$ consecutive points from the  sequence~$(\mu_k)$, that is the points $\mu_{n}, \mu_{n+1}, \dots, \mu_{n+j}$, we can construct the
Lagrange polynomial $L_{j,n}$ approximation of order $j$ for a function $f$, namely
\begin{equation}
L_{j,n}(\mu) = \sum_{i = 0}^j l_{j,n,i}(\mu) f(\mu_{n+i})
\end{equation}
where the function $l_{j,n,i}$ is the Lagrange basis polynomial
\begin{equation}
l_{j,n,i}(\mu) = \prod_{k = 0, k \ne i}^j \frac{\mu - \mu_k}{\mu_i - \mu_k}.
\end{equation}
These are defined for every function $f$, for every $n = 1, 2, \dots$ and for every $j = 1, 2, \dots$.

The $j$th order Lagrange polynomial $L_{j, n}$ coincides with the function $f$ (at least) in $j + 1$ points in the interval $I_n$ between $\mu_n$ and $\mu_{n + j}$, so by Rolle's Theorem the $j$th order derivative $L^{(j)}_{j, n}$ coincides with the $f^{(j)}$ at least in one point in the interval $I$. Therefore, by the Lagrange approximation, the derivatives of order $j$ of the error 
\begin{equation}
R_{j,n}(\mu) = f(\mu) - L_{j,n}(\mu)
\end{equation}
for a smooth function $f \in C^{j + 1}$ is bounded in the interval $I_n$ by
\begin{equation}
\sup_{\mu \in I_n} |\, \partial_\mu^{(j)}R_{j, n}(\mu)| \le \sup_{\mu \in I} \,| f^{(j + 1)}(\mu) \,| \, |\,\mu_n - \mu_{n + j} \,|.
\end{equation}
 Since the distance $|\, \mu_n - \mu_{n + j} \,| \to 0$ as $n \to \infty$, the $j$th derivative of the Lagrange polynomial $\partial_\mu^{(j)}L_{j, n}$ evaluated at $\mu_n$ is an approximation of the $j$th derivative of the function $f$ evaluated at $\mu_n$. Let $a_{n,j,i} = \partial_{\mu}^{(j)}l_{n,j,i}(\mu_n)$ be the corresponding coefficient in the expansion
\begin{equation}
(\partial_{\mu}^{(j)} L_{j, n})(\mu_n) = 
\sum_{i = 0}^{j} \partial_{\mu}^{(j)}l_{n,j,i}(\mu_n) f(\mu_{n + i})
= \sum_{i = 0}^{j} a_{n,j,i} f(\mu_{n + i})
\end{equation}
Note that the coefficients don't depend on the function $f$, but only on the indices~$n,i,j$ and the sequence~$(\mu_n)$.
Hence, we assume that the function $f(\mu) = e^{\mu p}$, which is a smooth function with random coefficient $p$ in the exponent. Moreover, the $j$th derivative of $f$ is $f^{(j)}(\mu) = p^j f(\mu)$.

Therefore,
 we have for every $n$ and every $j$ that
\begin{equation}
    \E \parens{ \rho L_{j, n}(\mu_n) }
    = \E \parens{ \rho \sum_{i = 0}^{j} a_{n,j,i} e^{p \mu_{n + 1}} }
    = \sum_{i = 0}^{j} a_{n,j,i} \widetilde D(\mu_{n + i})
    = 0
\end{equation}
since $\widetilde D(\mu) = 0$ for every $\mu \in V$.
Since the random variable $p$ is bounded, we may apply dominated convergence and we obtain
\begin{equation}
  \E \parens{ \rho p^j e^{\mu_0 p} }
  = \lim_{n \to \infty} \E \parens{ \rho L_{j + 1, n}(\mu_n) }
  = 0 
\end{equation}
for every $j = 0, 1, \dots$.
%
that implies that with linearity that we have $\E r(p) e^{\mu_0 p} \rho = 0$ for all polynomials~$r$.

The same approximation would be obtained with iterated differences, when these are defined recursively for sequences $\lambda = (f(\mu_n))_n$ as
\begin{equation}
\begin{split}
    d(1, \mu, \lambda)_n &= \frac{\lambda_{n+1} - \lambda_{n}}{\mu_{n + 1} - \mu_n}\\
    d(j + 1, \mu, \lambda)_n &= \frac{d(j, \mu, \lambda)_{n+1} - d(j, \mu, \lambda)_{n}}{\mu_{n + j + 1} - \mu_n}
\end{split}
\end{equation}
where the high-order differences use the right scaling. If the differences would form a grid, these would be the usual higher order stencil approximations. However, showing that this works in general is easiest to show via Rolle's Theorem and therefore, leads directly to the Lagrange approximation.

Now, by standard argument, since $p$ has compact range, the boundedness of $p$ implies by Stone-Weierstrass that for all continuous functions~$f$ we have $\E f(p) e^{\mu_0 p} \rho  = 0$.
Monotone convergence then implies that
\begin{align}
                & \E[p \in J]e^{\mu_0 p} \rho = 0 \text{ for all intervals } J              \\
    \implies    & \E\parens{ e^{\mu_0 p} \rho | p } = 0  \text{ holds } \Prob\text{-a.s.}   \\
    \iff        & e^{\mu_0 p} \E \parens{ \rho | p } = 0 \text{ holds }\Prob\text{-a.s.},
\end{align}
which implies claim~\ref{tfae:2} in the theorem.
\end{proof}

\section{Numerical demonstration}

To generate the synthetic measurement data~$D_j$, we compute the integrals 
\begin{equation}\label{eq:1}
\int_0^1 \lambda_j^{p(x)}\rho_0(x) \der x = D_j
\end{equation}
by using the Simpson $(2n+1)$-point quadrature rule, where $n=64$. In all of our examples, the measurements are corrupted by small additive Gaussian white noise with standard deviation~$\sigma$ equal to $0.5\,\%$ of the maximum of absolute value of the measurements.
Next, to solve the inverse problem of recovering~$\rho_0$ in~\eqref{eq:1} from the measurements~$D_j$ we first substitute instead of~$\rho_0$ the piecewise linear form
$$
\rho_0(x)\approx \sum_{k=1}^N f_k \phi_k(x),
$$
where~$\phi_k$ are the hat-functions
$$
\phi_k(x) =
\begin{cases}
\dfrac{x-x_{k+1}}{x_k-x_{k+1}},& x_k\leq x\leq x_{k+1}\\
\dfrac{x-x_{k-1}}{x_k-x_{k-1}},& x_{k-1}\leq x <x_k,\\
0,& \text{otherwise}.
\end{cases}
$$
Then~\eqref{eq:1} becomes
\begin{equation}\label{eq:2}
\sum_{k=1}^N f_k\int_0^1 \lambda_j^{p(x)}\phi_k(x)\der x = D_j+E_j, 
\end{equation}
where~$E_j$ is an error term which we ignore in the sequel.
In practise, the integrals appearing in~\eqref{eq:2} can be numerically precomputed for a given function~$p$.
Finally equation \eqref{eq:2} can be written as the linear system $Af=D$ and solved by some regularization method.
We avoid committing inverse crime~\cite{Mueller:Siltanen:2012} by generating the measurements~$D_j$ independently of the theory matrix~$A$.
Indeed, the measurements are obtained by integrating the precise model~\eqref{eq:1} and adding noise, while the theory matrix~$A$ is computed by only integrating the known function~$\lambda^p(\cdot)$ against~$\phi_k$, $k=1,\ldots,N$, while avoiding the use of the unknown~$\rho_0$.
Further, we shall use smaller number of measurements~$D_j$ than the number of unknowns~$f_k$.
In our case the number of measurements and unknowns will be 300 and 400, respectively. 

The computation is not very demanding and can be expected to work on a modern computer. The numerical work is done in \textsc{Matlab}.

\subsection{Regularization}
Since the properties of this problem depend quite drastically on the functions~$p$ and~$\rho$, at this point we opt to not propose any universal solution to the problem.
Rather, this part should be regarded as a demonstration that the problem can in principle be solved numerically.
As is usual in inverse problems, the linear problem $Af=D$ is rather unstable and regularization methods are necessary.
We have tested Tikhonov, total variation (TV), and conjugate gradient least squares (CGLS) regularization methods.
In Tikhonov regularization~\cite{Hansen:2010, Mueller:Siltanen:2012} one solves the minimization problem
$$
\mathrm{arg\,min}\left\{ \Vert Af - D\Vert_2^2 + \alpha \Vert f \Vert_2^2 \right\}.
$$
This classical regularization method is very simple to implement and in our tests works well with low noise-levels when the unknown~$\rho_0$ is reasonably smooth. The down-side to the $L^2$-penalty is that it  promotes smoother solutions, failing to recover discontinuous and irregular functions.
The TV-regularization instead aims to minimize the expression
$$
\mathrm{arg\,min}\left\{ \Vert Af - D\Vert_2^2 + \alpha \Vert f' \Vert_1 \right\}.
$$
The $L^1$-penalty term allows some steep gradients and thus can be used when $\rho$ is more irregular~\cite{Hansen:2010, Mueller:Siltanen:2012}.
Finally, the CGLS method is an iterative regularization procedure, where one solves the least-squares problem
$$
\mathrm{arg\,min}\left\{ \Vert Af_k - D\Vert_2^2\right\}
$$
subject to the condition that 
$$f_k\in \mathrm{span}\left\{A^TD, (A^TA)A^TD,\ldots,(A^TA)^{k-1}A^TD  \right\},\quad k=0,1,2,\ldots,$$
that is, the iterates belong to a Krylov subspace.
Heuristically, this method attempts to pick only the significant singular components of the solution~$f_k$, because the Krylov subspaces take into account the (noisy) data~$D$.
Using only the significant singular components allows one to reduce the effect of noise.
We use the CGLS-algorithm of Hestenes and Stiefel; see e.g.~\cite{Hansen:2010, Hestenes:Stiefel:1952}.

\subsection{Examples}
Let the number of measurements be $M=300$ and the number of unknown coefficients of the hat functions be $N=400$.
In these examples the measurement points~$\lambda_j$ are uniformly distributed on the interval $(0,1)$.
We also tested different distributions of~$\lambda_j$ on the interval $(0,1)$, but often a good result was obtained with uniform distribution. Certainly this choice depends on the function~$p$ and can be tailored for applications separately.
The functions used in the examples are as follows:
\begin{itemize}
\item Example 1: $p(x)=x$ and $\rho_0(x)=\sin(\pi x)$.
\item Example 2: $p(x)=x$ and $\rho_0(x)=0.3\chi(x)$, where $\chi(x)$ is the characteristic function of the inverval $(0.3,0.6)$.
\end{itemize}

The results are presented in Figure \ref{fig:ex1-2}, where the precise unknown~$\rho_0$ is depicted with the red dashed line and the numerically computed solutions~$\rho$ are shown as the solid blue line.
Apparently Tikhonov- and CGLS-solutions work rather well with smooth~$\rho_0$, while in the discontinuous case the TV-regularized solution is considerably better.
\begin{figure}[t]
\centering
\includegraphics[width=0.3\textwidth]{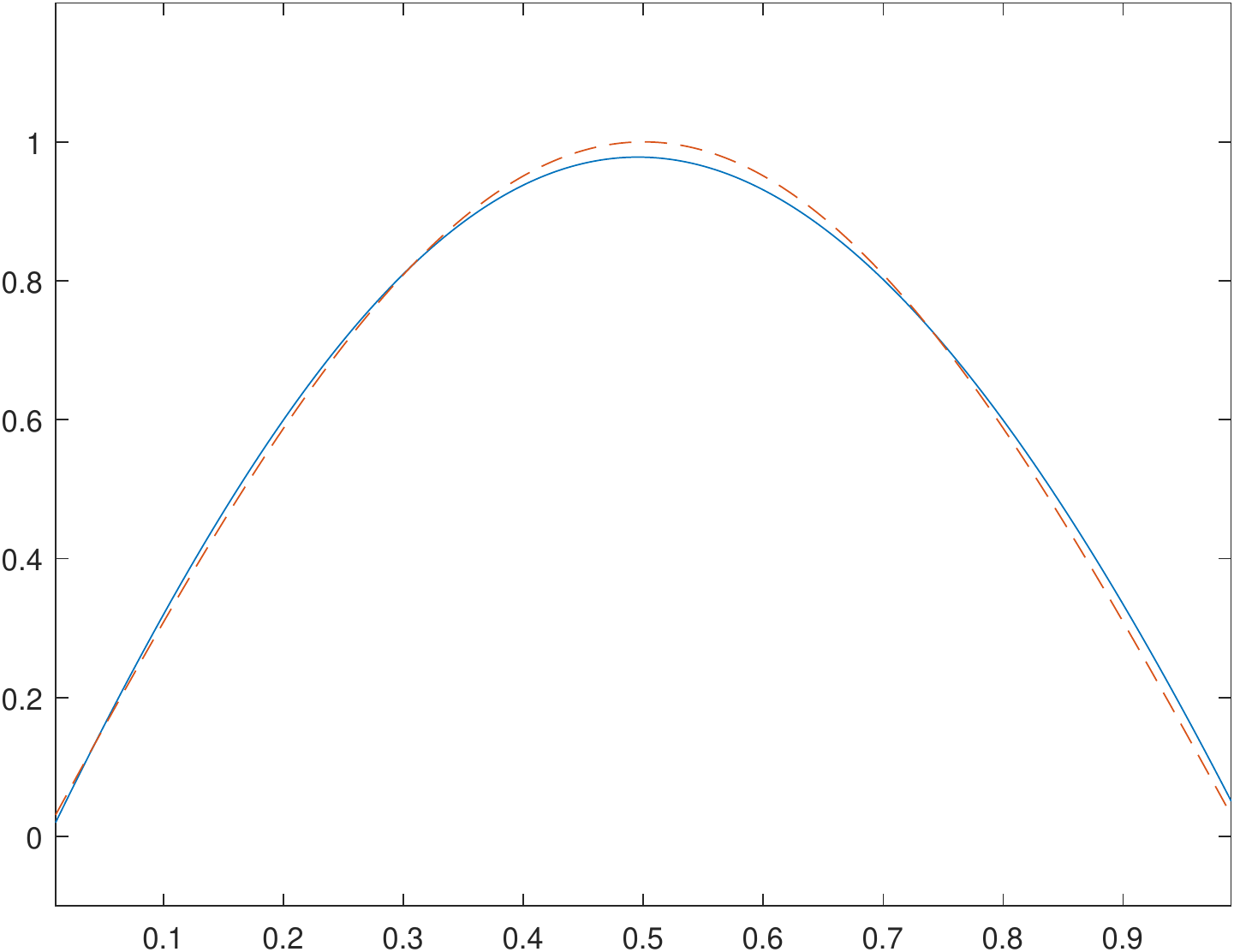}
\includegraphics[width=0.3\textwidth]{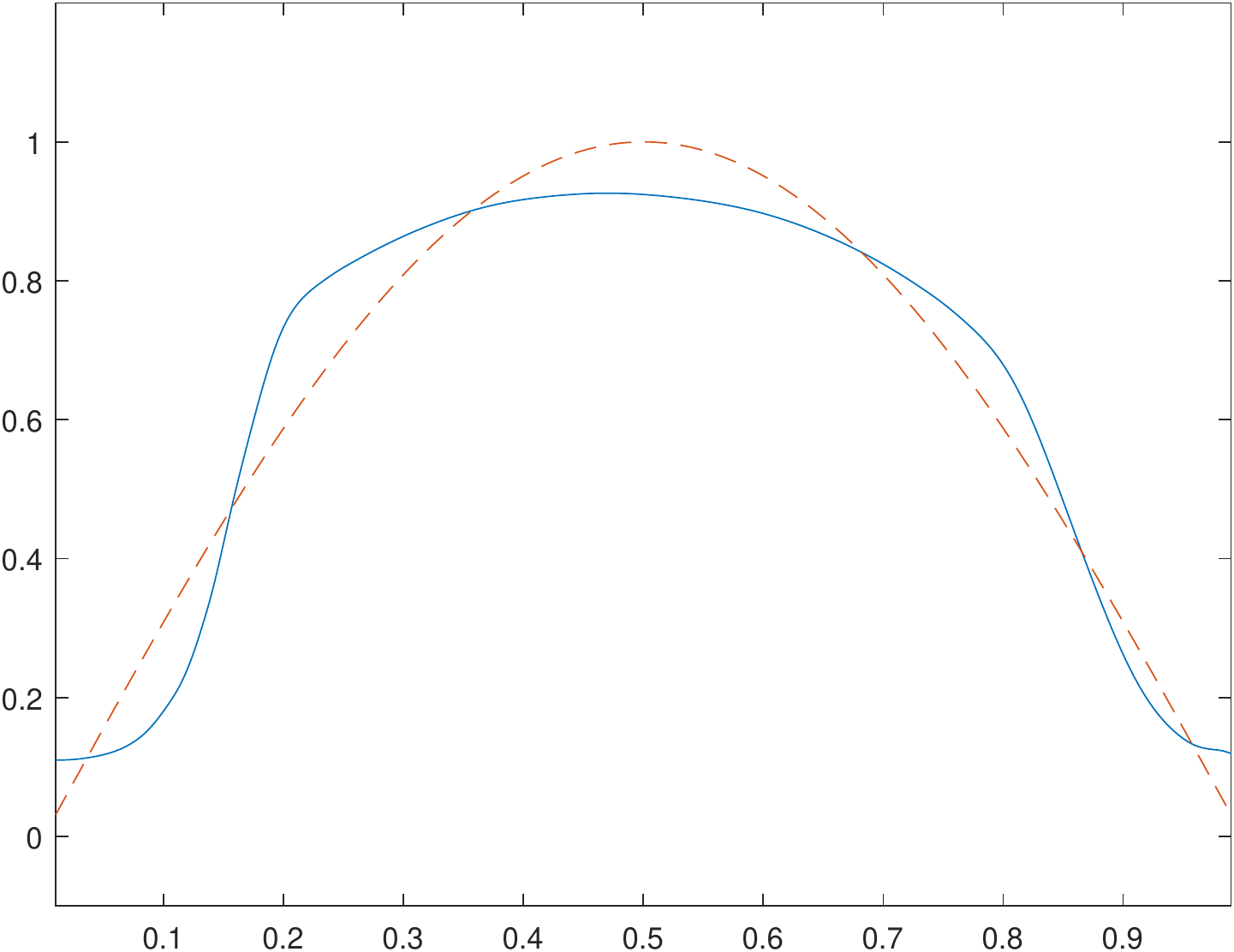}
\includegraphics[width=0.3\textwidth]{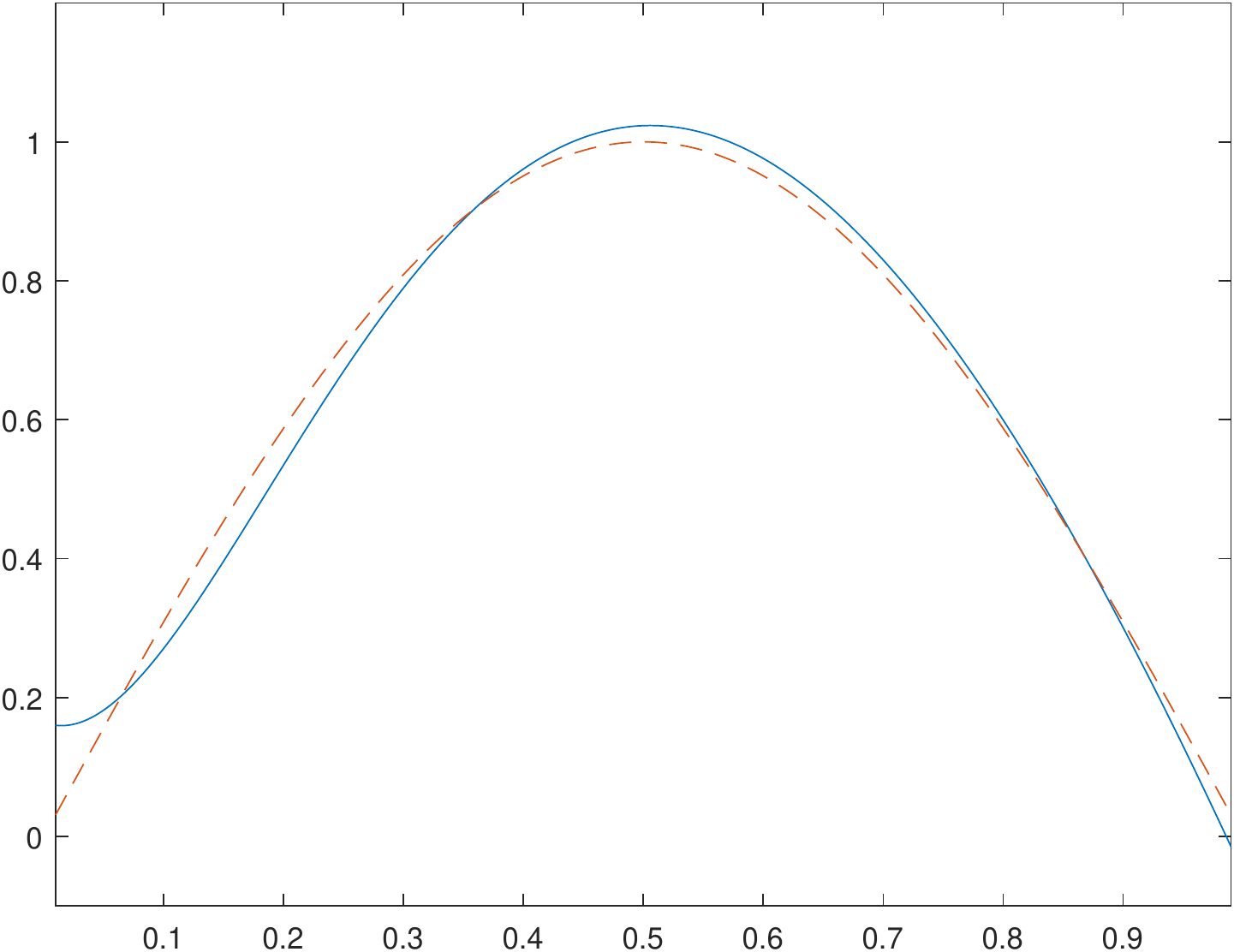}
\\
\vspace*{0.2cm}
\includegraphics[width=0.3\textwidth]{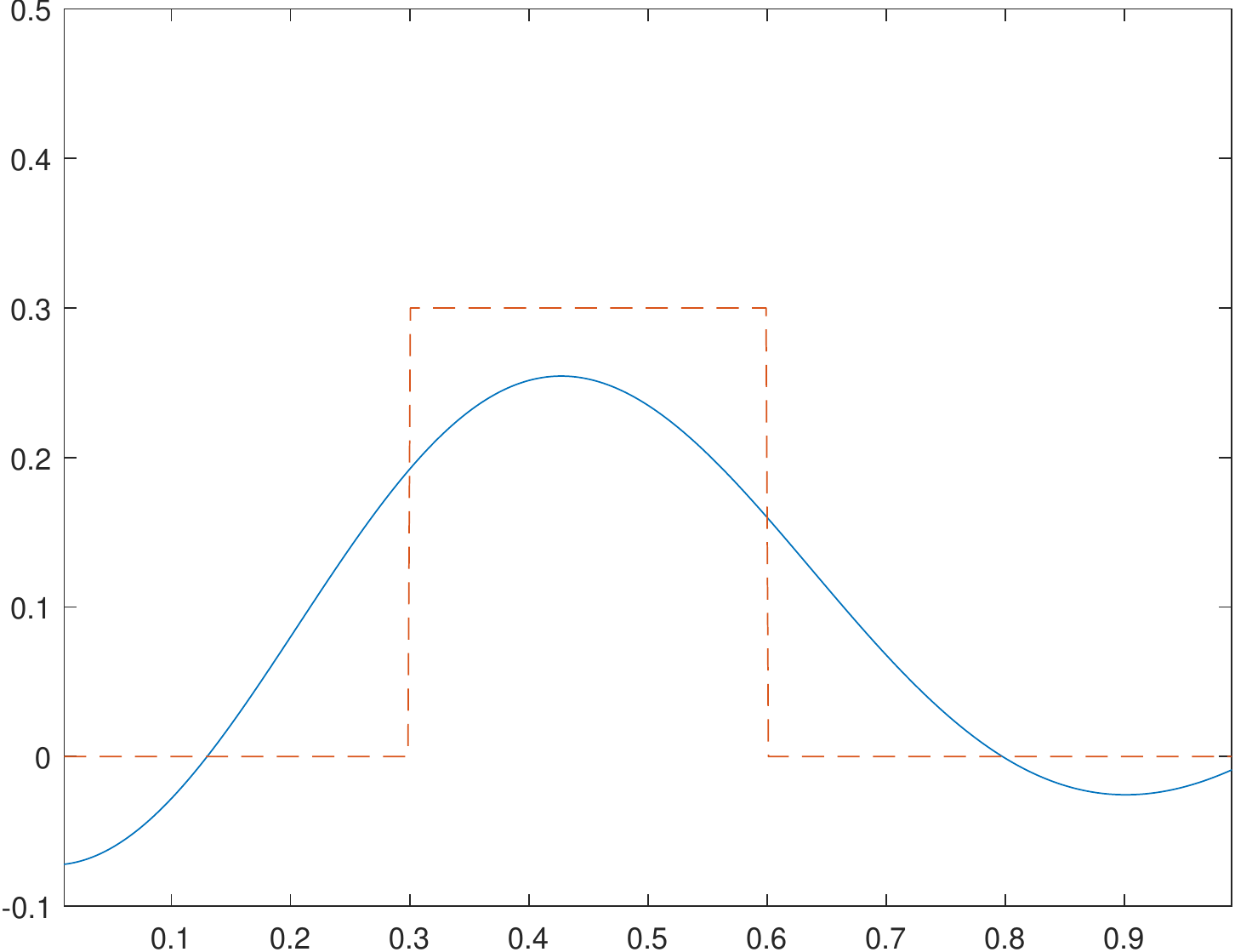}
\includegraphics[width=0.3\textwidth]{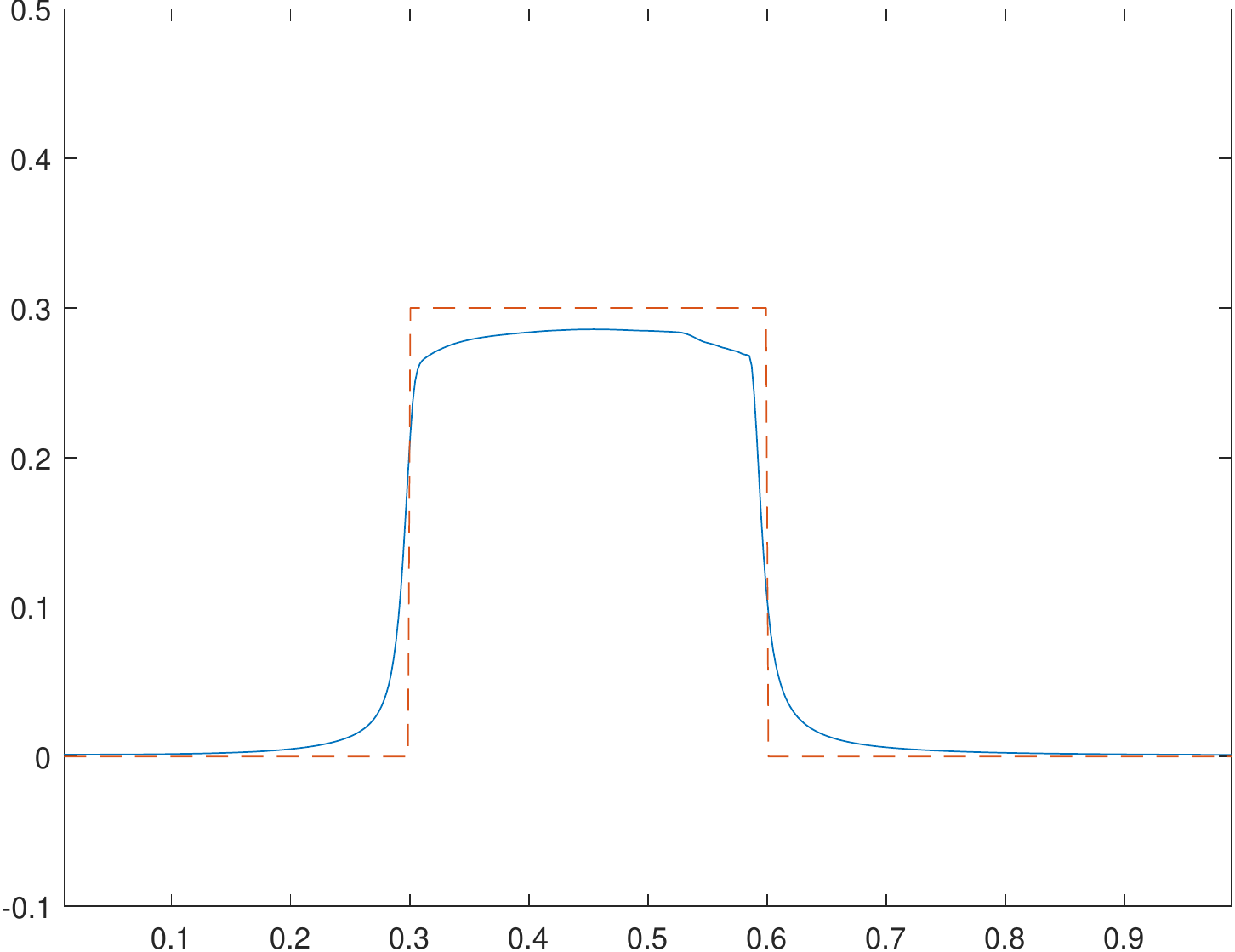}
\includegraphics[width=0.3\textwidth]{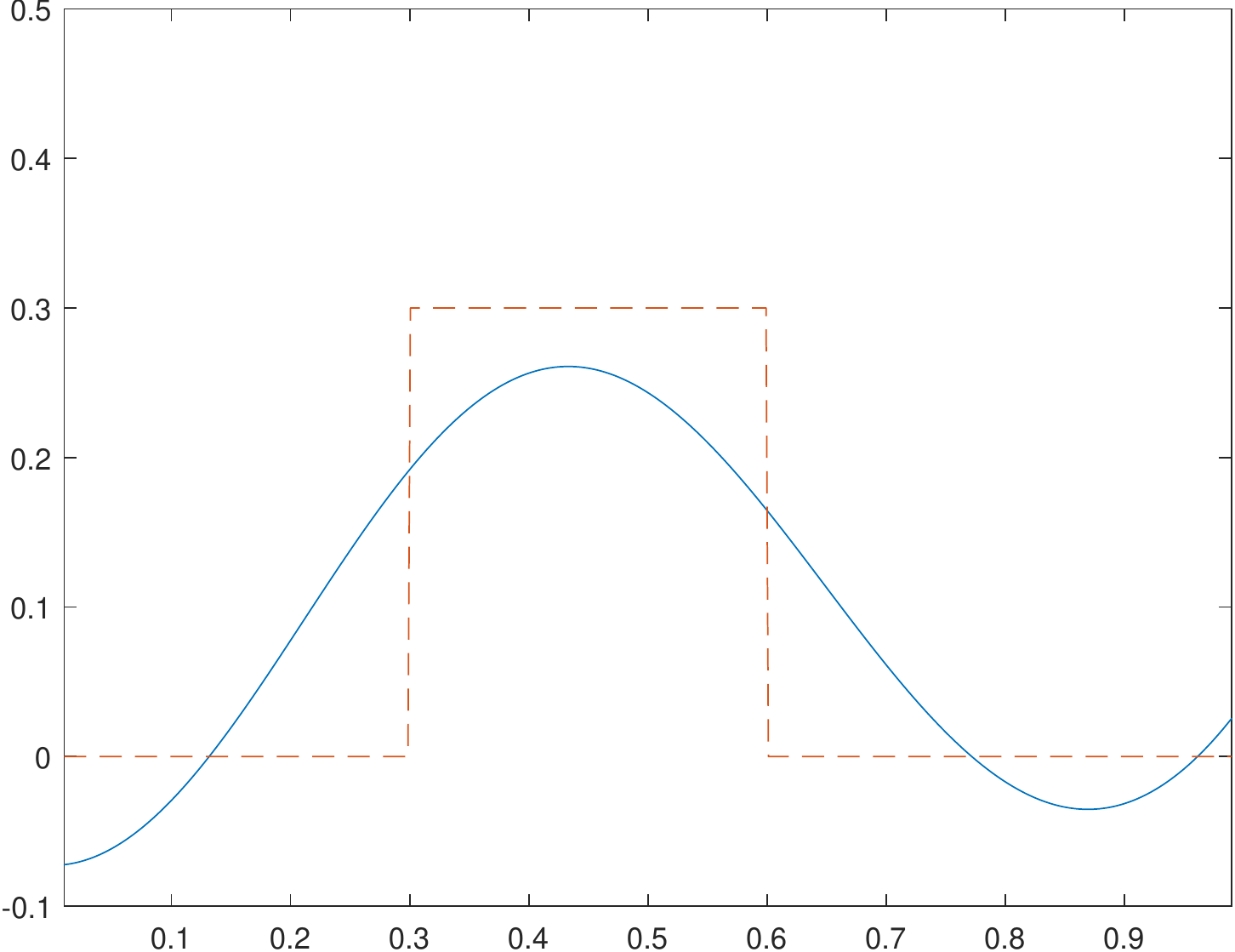}
\caption{Unknown $\rho_0$ (dashed red line) and the numerical solution $\rho$ (solid blue line) with 0.5\% noise level. Example 1 (above) and example 2 (below) with Tikhonov-solution (left), TV-solution (middle) and CGLS-solution (right).\label{fig:ex1-2}}
\end{figure}

\subsection{The sets where $p$ is constant}

If the function~$p$ is constant in some set of positive measure, the theory predicts that we can only hope to solve~\eqref{eq:1} up to the average of~$\rho_0$ in that set.
Let $M=300$, $N=400$, and $\rho_0(x)=\sin(\pi x)$.
The function~$p$ is given in the piecewise form
$$
p(x)=
\begin{cases}
5/3x,&\quad x\in(0,0.2],\\
1/3,&\quad x \in (0.2,0.4],\\
5/3x-1/3,&\quad x \in (0.4,0.6],\\
2/3,&\quad x\in (0.6,0.8],\\
5/3x-2/3,&\quad x \in (0.8,1].
\end{cases}
$$
Here the optimal information to recover is $\sin(\pi x)$ on the intervals $(0,0.2]$, $(0.4,0.6]$, and $(0.8,1]$ and only the average value of $5/2\pi\approx 0.796$ on the intervals $(0.2,0.4]$ and $(0.6,0.8]$.
The results are depicted in Figure~\ref{fig:ex3}, where the unknown~$\rho_0$ is shown in red dashed line, the projection~$P\rho_0$ in black dot-dash line and the numerical solution in solid blue line.

The numerical method is expected to produce a function whose projection (averages over sets where~$p$ is constant) is~$P\rho_0$.
There are many such functions, and the choice depends on regularization.
With Tikhonov one expects to find the function with minimal~$L^2$ norm -- which is precisely~$P\rho_0$ -- but with other regularizations something else.

\begin{figure}[ht]
\centering
\includegraphics[width=0.3\textwidth]{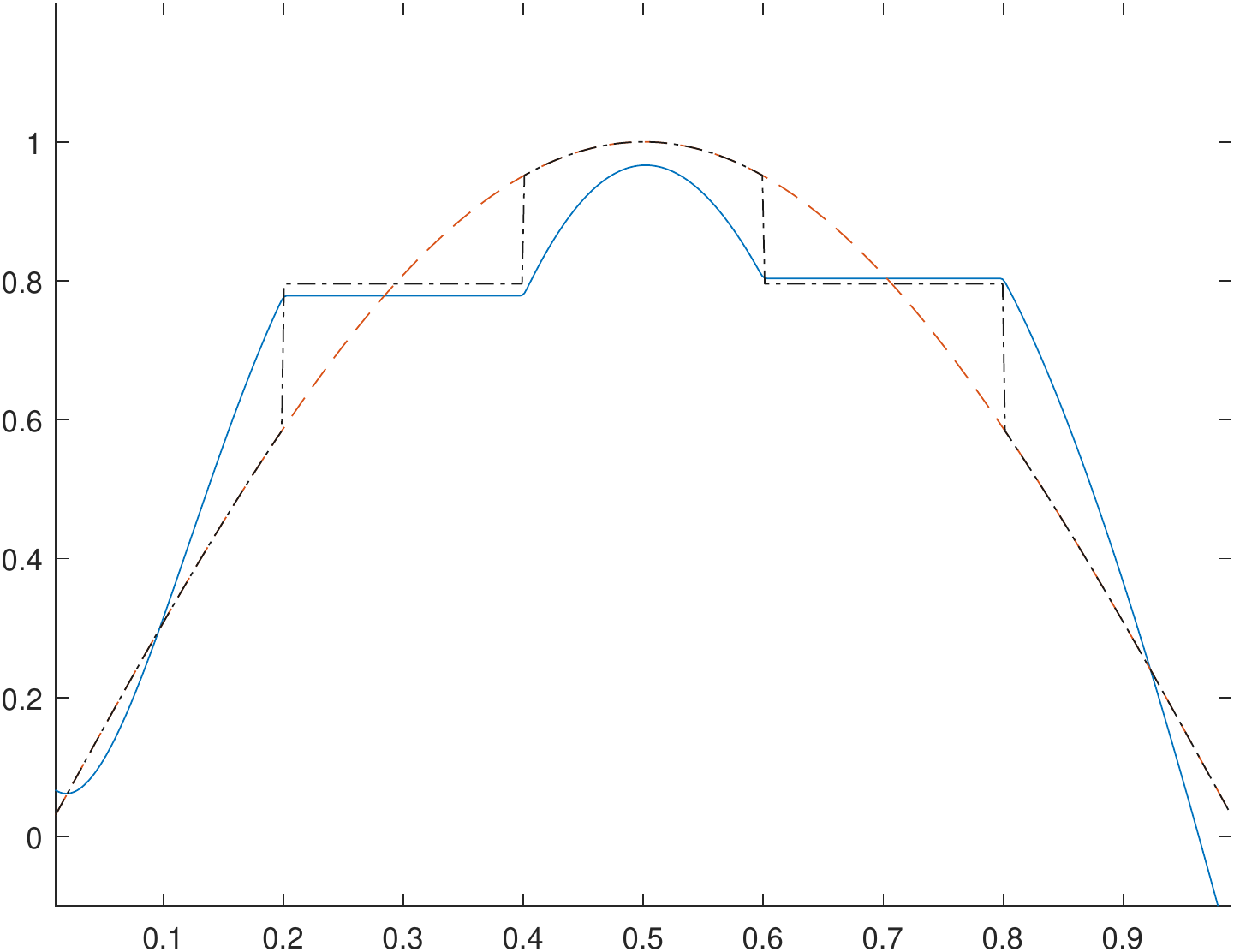}
\includegraphics[width=0.3\textwidth]{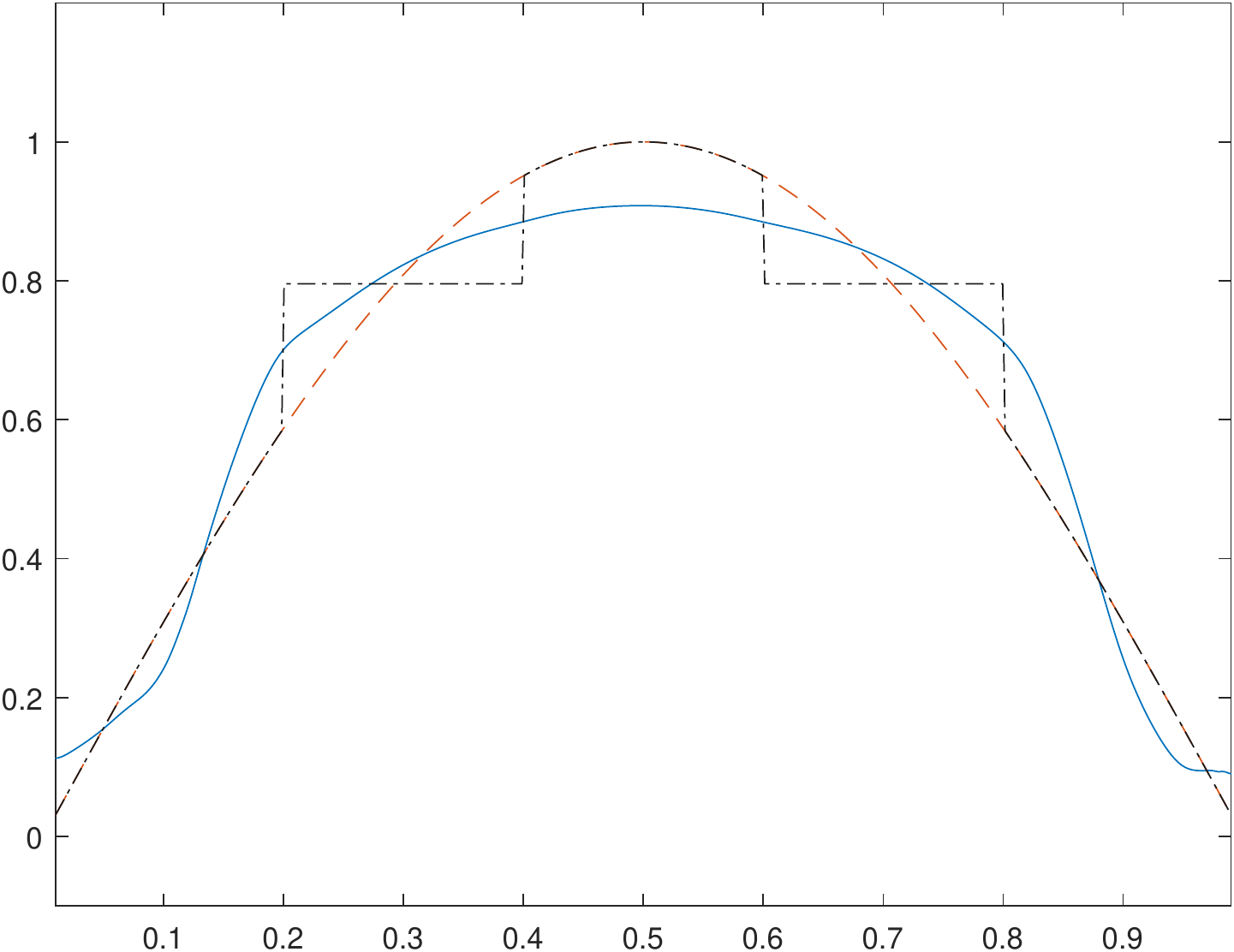}
\includegraphics[width=0.3\textwidth]{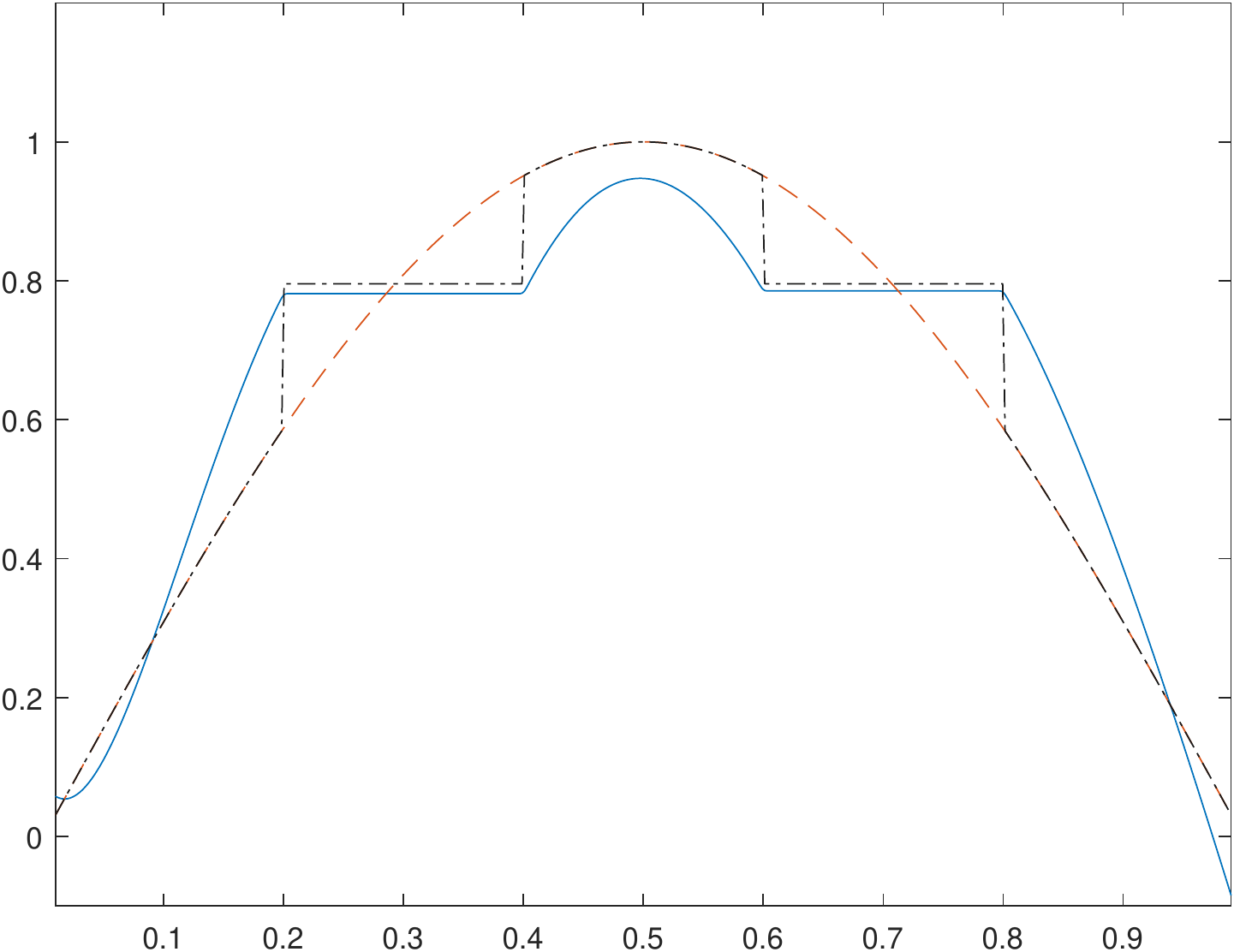}
\caption{Unknown $\rho_0$ (dashed red line), $\rho_0$ averaged over regions where $p$ is constant (black dot-dash line) and the numerical solution $\rho$ (solid blue line) with 0.5\% noise level. Tikhonov-solution (left), TV-solution (middle) and CGLS-solution (right).\label{fig:ex3}}
\end{figure}

\subsection{Limited data}
The theory states that it suffices to have measurements for $\lambda_j$'s in some open interval $(a,b)\subset (0,1)$.
Let $M=300$, $N=400$, the source density~\mbox{$\rho_0(x)=-0.5\sin(2\pi x)+0.5$}, and $p(x)=e^x-1$.
To test the solutions, we used several smaller intervals, solved the problem with Tikhonov regularization one hundred times, and collected the averaged $l^2$-relative errors in the solutions to Table~\ref{table:intervals}. The respective solutions~$\rho$ are depicted in Figure~\ref{fig:ex6}. These solutions are relatively good for all but the smallest interval $(0.4,0.5)$, where the numerical solution is rather unstable. Note that the number of measurements is kept at constant 300 for all solutions. We state without details that similar results are obtained if the measurements are made in some union of open intervals (with large enough measure) contained in $(0,1)$.
\begin{table}[ht]
\centering
\begin{tabular}{|c|c|c|c|c|}
\hline
Intervals& $(0,1)$ & $(0.2,0.8)$ & $(0.3,0.6)$ & $(0.4,0.5)$  \\
\hline
$\overline{\epsilon}_\mathrm{rel}$ & 0.117 & 0.170 & 0.186 & 0.320  \\
\hline
$\mathrm{var}$ & $1.88\cdot 10^{-3}$ & $1.27\cdot 10^{-2}$ & $1.27\cdot 10^{-2}$ & $6.65\cdot 10^{-3}$ \\
\hline
\end{tabular}
\caption{Averaged relative errors and variances of one hundred solutions on smaller intervals with noise level 0.5\%. \label{table:intervals}}
\end{table}

\begin{figure*}[ht]
\centering
\begin{subfigure}[t]{0.45\textwidth}
\includegraphics[width=\textwidth]{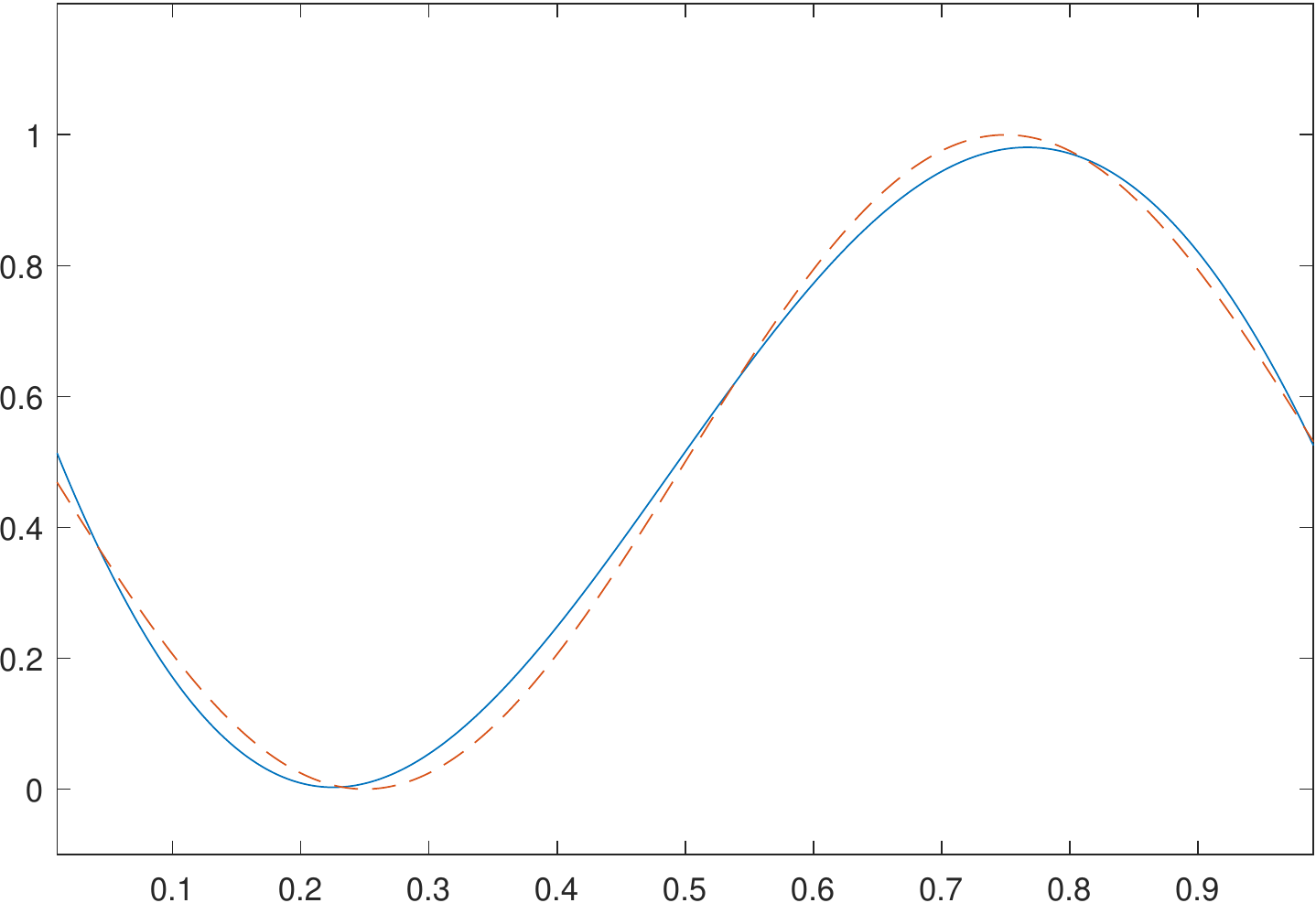}
\caption{Measurement interval $(0,1)$}
\end{subfigure}
\begin{subfigure}[t]{0.45\textwidth}
\includegraphics[width=\textwidth]{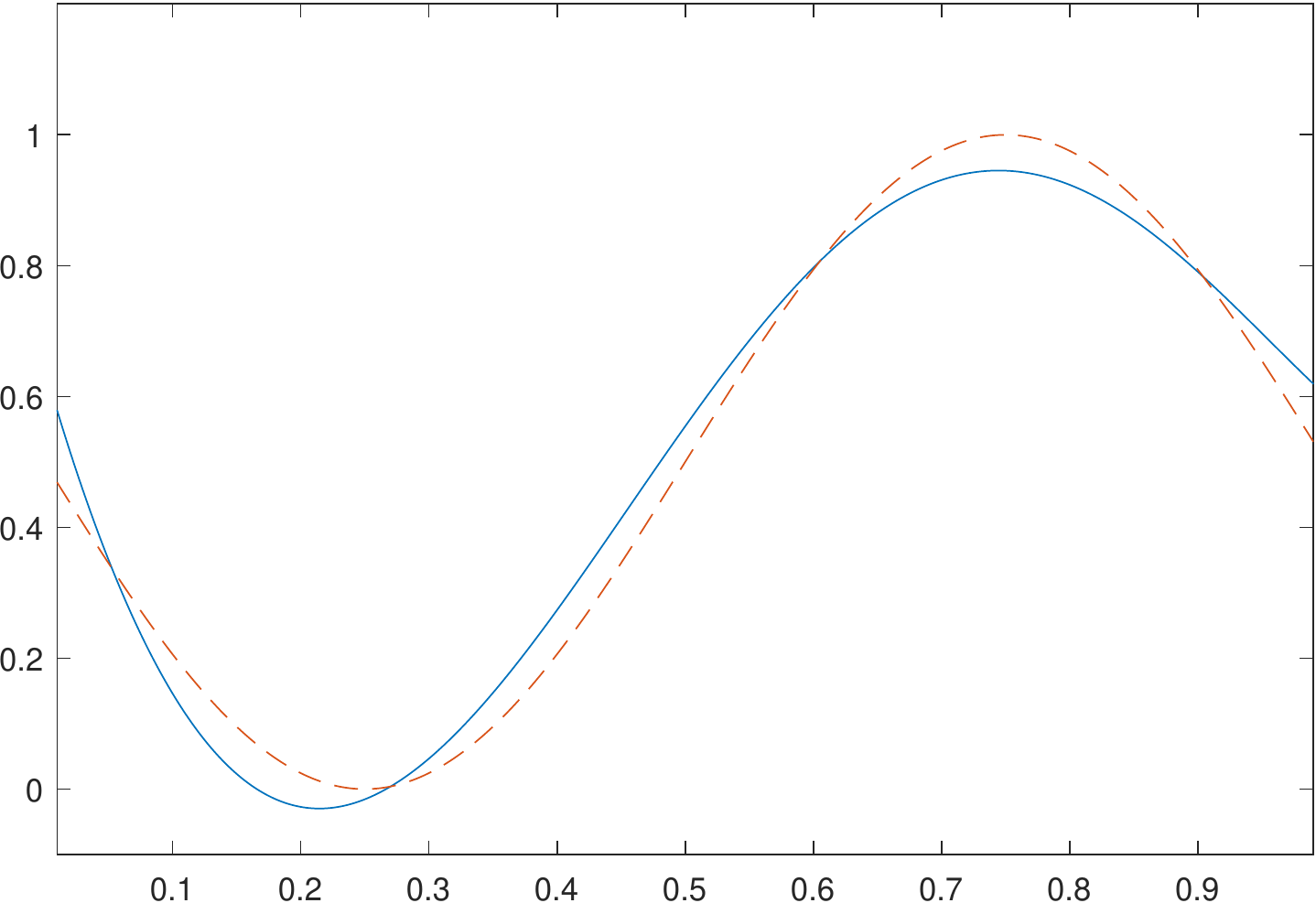}
\caption{Measurement interval $(0.2,0.8)$}
\end{subfigure}\\
\vspace*{0.3cm}
\begin{subfigure}[t]{0.45\textwidth}
\includegraphics[width=\textwidth]{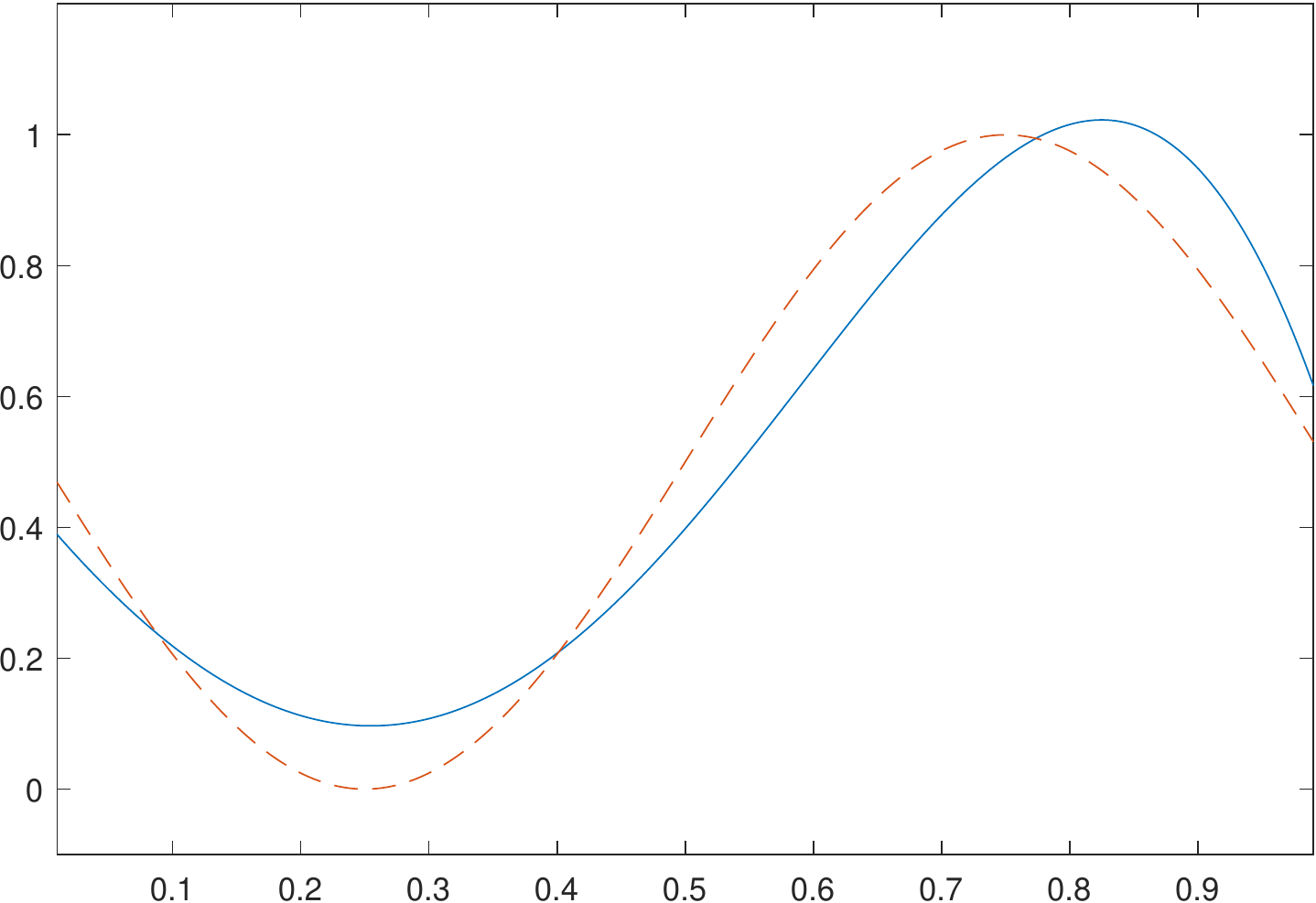}
\caption{Measurement interval $(0.3,0.6)$}
\end{subfigure}
\begin{subfigure}[t]{0.45\textwidth}
\includegraphics[width=\textwidth]{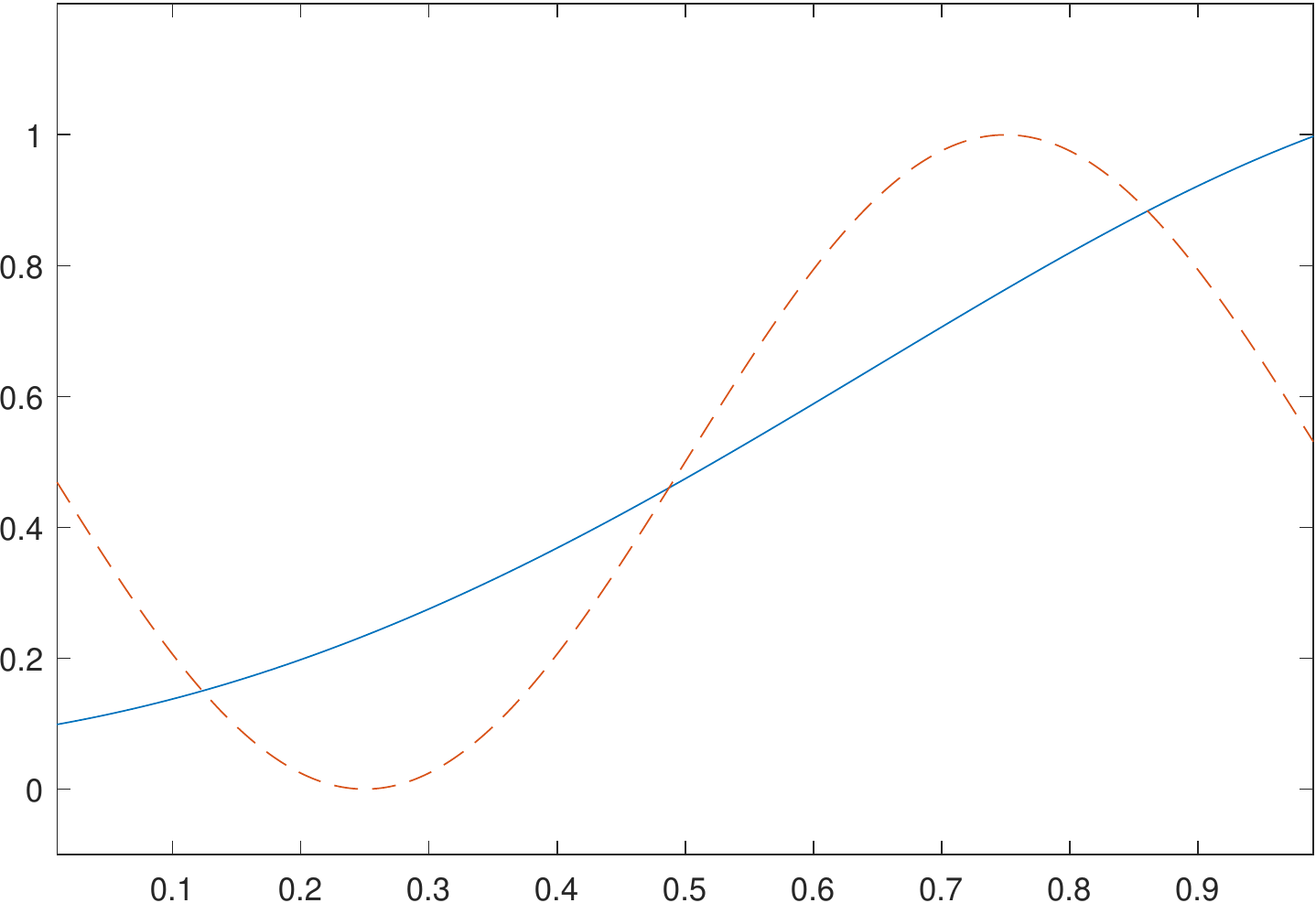}
\caption{Measurement interval $(0.4,0.5)$}
\end{subfigure}
\caption{Unknown $\rho_0$ (dashed red line) and the Tikhonov-solution $\rho$ (solid blue line) with 0.5\% noise level and smaller measurement intervals. \label{fig:ex6}}
\end{figure*}

\bibliographystyle{plain}
\bibliography{math}

\end{document}